\documentclass{article}

\usepackage{tikz}
\usetikzlibrary{arrows,shapes}

\usepackage{microtype}
\usepackage{graphicx}
\usepackage{subfigure}
\usepackage{booktabs} 

\usepackage{hyperref}

\usepackage[accepted]{icml2024}

\usepackage[utf8]{inputenc}
\usepackage{amssymb}
\usepackage{verbatim}
\usepackage{amsmath}					
\usepackage{amssymb}					
\usepackage{mathrsfs}					
\usepackage{amsthm}
\usepackage{graphicx}

\newcommand{\R}{\mathbb{R}}

\newcommand{\pr}{\mathbf{P}}

\newcommand{\w}{\mathbf{w}}

\newcommand{\T}{\intercal}

\theoremstyle{plain}

\newtheorem{theorem}{Theorem}

\newtheorem{remark}{Remark}

\icmltitlerunning{Deep Reinforcement Learning: A Convex Optimization Approach}

\begin{document}
	
	\twocolumn[
	\icmltitle{Deep Reinforcement Learning: A Convex Optimization Approach}
	
	
	
	
	\begin{icmlauthorlist}
		\icmlauthor{Ather Gattami $\langle$ atherg@gmail.com $\rangle$ }{}
	\end{icmlauthorlist}
	
	

	\icmlkeywords{Deep Reinforcment Learning, Convex Optimization, Episodic $Q$-Learning}
	
	\vskip 0.3in
	]
	
	
	

	\begin{abstract}
		In this paper, we consider reinforcement learning of nonlinear systems with continuous state and action spaces. We present an episodic learning algorithm, where we for each episode use convex optimization to find a two-layer neural network approximation of the optimal $Q$-function. The convex optimization approach guarantees that the weights calculated at each episode are optimal, with respect to the given sampled states and actions of the current episode. For stable nonlinear systems, we show that the algorithm converges and that the converging parameters of the trained neural network can be made arbitrarily close to the optimal neural network parameters. In particular, if the regularization parameter in the training phase is given by $\rho$, then the parameters of the trained neural network converge to $w$, where the distance between $w$ and the optimal parameters $w^\star$ is bounded by $\mathcal{O}(\rho)$. That is, when the number of episodes goes to infinity, there exists a constant $C$ such that
		\[
		\|w-w^\star\| \le C\rho.
		\]	
		In particular, our algorithm converges arbitrarily close to the optimal neural network parameters as the regularization parameter goes to zero. As a consequence, our algorithm converges fast due to the polynomial-time convergence of convex optimization algorithms.
	\end{abstract}
	
	\section{Introduction}
	\subsection{Background}
	Deep Reinforcement Learning (RL) has been a cornerstone in most recent developments of Artificial Intelligence. One example was  defeating the highest ranked player in the ancient game Go \cite{silver:2017}, which in the 90's was considered a challenge that is hard to crack for the coming 100 years. Another example is the development of chat bots with Large Language Models, based on RL with Human Feedback.
	
	Most of the recent practical progress is related to Markov Decision Processes (MDPs) with discrete state and/or action spaces. The case of continuous state and action spaces is hard in general due to the high complexity of the problem. 
	
	The optimal control in its full generality is given by a dynamical system
	\[
		x_{t+1} = f(x_t, u_t)
	\]
	where $x_t$ represents the state and $u_t$ the controller, at time step $t$. The goal of the controller is to minimize a certain criterion, is to minimize the average cost
	$$
	\lim_{N\rightarrow \infty} \frac{1}{N} \sum_{t=1}^{N} c(x_t, u_t)
	$$
	or the discounted cost
	$$
	\lim_{N\rightarrow \infty}
	\sum_{t=1}^{N} \gamma^{t}c(x_t, u_t)
	$$
	for $\gamma \in (0,1)$, subject to some power constraint on the control signal, $u\in \mathcal{U}$, typically given by a bound on its power, $|u_t|\le c_u$.
	
	Bellman provided a general algorithm to solve the problem of optimal control based on dynamics programming, given by the well known Bellman equation
	\begin{equation*}
		V_t(x) = \min_{u}{\left\{c(x,u) + \gamma V_{t+1}(f(x,u))\right\}}
	\end{equation*}
	The challenge with the above equation is what Bellman referred to as the "curse of dimensionality", where for most of the cases, the dynamic programming solution explodes exponentially in the dimensions of the state and action spaces, and the length of the time horizon. Therefore, most of the known approaches are approximate even for the case where the system parameters and the cost function are known.

	\subsection{Related Work}
	Nonlinear control theory is mainly considered with stabilizing nonlinear systems, relying on different approaches, see for instance \cite{Khalil:2002}. The Bellman equation has been a standard tool for optimal control. For continuous state and action spaces, a straight forward approach is to discretize the state/action spaces and then use existing solutions for discrete MDPs such as different variants of $Q$-learning, see \cite{sutton:1998} for an overview. However, the problem with this approach is that the higher the resolution is, the larger state and action spaces become. This will in turn increase the computational complexity of the problem. 
	To tackle the computational complexity, methods relying on function approximations are used. A relatively simple approach from a complexity point of view is to use linear function approximation \cite{melo:2007}, where convergence is shown, given certain conditions that could be too restrictive. Most recently, function approximation based on Neural Networks have been widely used, due to its success in the case of discrete state and action spaces. However, the current methods suffer from several drawbacks. First, there are no convergence guarantees when the $Q$ function is approximated with a neural network. Second, even if the algorithms converge, it's not clear how far from the optimum they converge to. We refer the reader to \cite{Hasselt:2012} for a more thorough literature review of Reinforcement Learning in continuous state and action spaces. For the case of discounted cost, and discrete state and action spaces, convergence is shown for any number of layers in \cite{sun:2022}. Also, for the case of discounted cost with possibly continuous state spaces and discrete action spaces, the authors in \cite{gaur:2023} introduce a converging algorithm where the value function is approximated by two layer neural networks, relying on the results by \cite{pilanci:2020}, where training two layer neural networks can be transformed to a convex optimization problem.
	
	\begin{algorithm}[t]
		\label{alg}
		\caption{Episodic Learning with Convex Optimization}
		\begin{algorithmic}[1]
			\STATE Input $\gamma, \rho, R$
			\STATE Initialize $u_1, ..., u_T$
			\STATE Sample $D$ by running $u_1, ..., u_T$
			\STATE Initialize $w$
			\STATE Set $w_1 = w$
			\FOR {episode $k = 1, ..., K$:}
			\STATE 	Observe $x_1$
			\STATE Set $u_1 = \arg\min_{u} Q(x_1,u, {w_k})$
			\FOR{($t=1, ..., T$):}
			\STATE Apply $u_t$ and observe $x_{t+1}$
			\STATE Set $u_{t+1} = \arg\min_{u} Q(x_{t+1},u, w_k)$
			\STATE $X_t = 
			\begin{pmatrix}
				1 & x_t^\intercal & u_t^\intercal 
			\end{pmatrix}
			$
			\STATE $y_t = c(x_t,u_t) + \gamma Q(x_{t+1}, u_{t+1}, w_k)$
			\ENDFOR
			\STATE Solve (\ref{cvxopt0}) 
			and obtain the solution $w$ 
			\STATE $w_{k+1} \leftarrow w_{k} + 
			\alpha_k(w - w_{k})$
			\ENDFOR
		\end{algorithmic}
	\end{algorithm}\vspace{5mm}
	\subsection{Contributions}
	Our  main contribution is the introduction of Algorithm 1, where we episodically use convex optimization to find two-layer neural network approximation to the optimal $Q$-function for infinite horizon average and discounted reward cases. The convex optimization approach guarantees that the weights calculated at each episode are optimal, with respect to the given sampled states and actions of the current episodes. We show that the algorithm converges for stable nonlinear systems, and that the converging parameters of the trained neural network can be made arbitrarily close to the optimal neural network parameters. In particular, if the regularization parameter that is used in the convex optimizaiton procedure is given by $\rho$, then the algorithm parameters $w$  distance from the optimal parameters $w^\star$ is bounded by $\mathcal{O}(\rho)$. Hence, there is a constant $C$ such that
	\[
		\|w-w^\star\| \le C\rho.
	\]
	That is, by decreasing the regularization parameter, we can get arbitrarily close to the optimal parameters. As a consequence, our algorithm converges fast due to the polynomial-time convergence of convex optimization algorithms.
	Finally, we demonstrate the performance of our proposed algorithm with numerical experiments for a nonlinear dynamical system under power constraints on the control signal, where we could get very close to the universally optimal controller.

	\subsection{Notation}
	\begin{tabular}{ll}
		$\mathbb{N}$ & The set of positive integers.\\
		$\mathbb{R}$ & The set of real numbers.\\
		$\pr_{S}(~\cdot~)$ & $\pr_S(x)$ is the projection of $x\in \mathcal{X}$\\ & on the space $S$.\\
		$A_{i}$ & $A_{i}$ denotes ith $i$:th row of the matrix $A$.\\
		$[A]_{ij}$ & Denotes the element of the matrix $A$\\
		&  in position $(i.j)$\\
		$A_{i,j}$ & $A_{i,j}=[A]_{ij}$\\
		$\|\cdot\|_F$ & $\|A\|$ denotes the Frobenius norm\\
		& of the matrix $A$.\\
		$\|\cdot\|$ & $\|A\|$ denotes the $\infty$-norm of  the matrix $A$.\\
		$(~\cdot~)_+$ & For a vector $x\in \mathbb{R}$, $(x)_+ = v$, where \\
		& $v_i = \max(x_i, 0)$.\\
		$x_{+}$ & For a sample $x = x_t$, we have $x_{+} = x_{t+1}$. \\
		$w_{-}$ & For a sample $w = w_k$, we have $w_{-} = w_{k-1}$. 
		\end{tabular}

	\section{Episodic Deep Reinforcement Learning with Convex Optimization}
	
	Consider a dynamical system given by
	\begin{equation}
		\begin{aligned}
			x_{t+1}&= f(x_t, u_t)\\
		\end{aligned}
	\end{equation}
	where $x_t\in \mathbb{R}^n$ and $u_t\in \mathbb{R}^m$.
	Suppose that 
	the cost function $c(x, u)$ is non-negative and that it's bounded by some constant $\bar{c} \ge c(x,u)$, for all stabilizing control signals $u$. The goal of the controller $u_t = \mu_t(x_t, x_{t-1}, ..., x_1)$ is to minimize the average cost
	$$
	\lim_{N\rightarrow \infty} \frac{1}{N} \sum_{t=1}^{N} c(x_t, u_t)
	$$
	or the discounted cost
	$$
	\lim_{N\rightarrow \infty}
	\sum_{t=1}^{N} \gamma^{t}c(x_t, u_t)
	$$
	for $\gamma \in (0,1)$, subject to some power constraint on the control signal,  $|u_t|\le c_u$.
	
	Bellman's equation gives the recursive relation 
	\begin{equation}
		V(x) = \min_{u}{\left\{c(x,u) + \gamma V(f(x,u))\right\}}
	\end{equation}
	for $\gamma \in (0,1]$, where the case $\gamma = 1$ corresponds to the average cost case.
	and $V(x)$ is the value function. Alternatively, we can use the expression
		\begin{equation}
				\begin{aligned}
						Q(x,u) 	  &= c(x,u) + \gamma \min_{u_+} Q(x_+, u_+)
					\end{aligned}		
		\end{equation}
The $Q$ function can be approximated by a neural network as in \cite{silver:2017}. We let the activation functions in the neural network be given by the ReLU function 
$$(x)_+ \triangleq \max(x,0).$$  
	
Now suppose $Q$ is represented by a two layer neural network with layer parameters 
$$\w = (w, w'), ~~ w\in \R^{d\times M}, ~~ w'\in \R^{M},$$
\[
\begin{aligned}
	Q(x, u, \w) &= \left([1~~ x^\T~ u^\T]w\right)_+ w' \\
	&= \sum_{i=1}^{M} \left([1~~ x^\T~ u^\T]w_i\right)_+ w'_i.
\end{aligned}
\]

Let $\mathbf{w}^\star$ be the optimal network parameters. Then, they satisfy the Bellman equation

\begin{equation}
	\label{scaler_bellman}
	\begin{aligned}
		Q(x,u, \w^\star) 	  &= c(x,u) + \pr_{\mathbf{Q}}\left(\gamma \min_{u_+} Q(x_+, u_+, \w^\star)\right)
	\end{aligned}		
\end{equation}
where 
\[
	\mathbf{Q} = \left\{Q: Q = (x^\T w)_+ w', w\in \R^{d\times M}, w' \in \R^M, M\in \mathbb{N}  \right\}
\]	
	The training of the $Q$ function parameters over a batch of $T$ time steps is given by minimizing the squared error function with respect to the weights $\mathbf{w}$ of a neural network, that is minimizing the loss
	
	\[
	\min_{\mathbf{w}} \sum_{t=1}^{T} l(\mathbf{w})
	\]
	where	
	\[
	\begin{aligned}
	l(\mathbf{w}) &= \left(c(x,u) + \gamma \min_{u_+} Q(x_+, u_+, \mathbf{w_-}) - Q(x,u,\mathbf{w})\right)^2\\
	&~~~~~~ + \rho_T R(\mathbf{w}),
	\end{aligned}
	\]
	$R(\mathbf{w})$ is some regularization term that can be chosen appropriately, and $\rho_T = \rho T$ for some real number $\rho > 0$.

%
	
	The minimization with respect to the loss function $l(\mathbf{w})$ is not necessarily a convex optimization problem, and it could be hard to find the right neural network approximation. To get around this problem, we will consider a so called episodic setting, which we will describe here. Consider learning to control the dynamical system over $K$ episodes, where each episode has a time horizon of $T$ time steps. The input data will be given by 
	\[
	X = \begin{pmatrix}
		1 & x_1^\intercal & u_1^\intercal \\
		1 & x_2^\intercal & u_2^\intercal \\
		\vdots & \vdots & \vdots\\
		1 & x_{T-1}^\intercal & u_{T-1}^\intercal \\
		1 & x_T^\intercal & u_T^\intercal 
	\end{pmatrix}
	\]
	and output data will be given by
	$$y_t = c(x_t,u_t) + \gamma Q(x_{t+1}, u_{t+1}, \mathbf{w}_k),$$ for $t=1, ..., T$. 
	
	Now, let 
	\[
	c = 
	\begin{pmatrix}
		c(x_1, u_1) \\
		c(x_2, u_2) \\
		\vdots\\
		c(x_T, u_T)
	\end{pmatrix}	
	\] 
	and 
	\[
	Z = \begin{pmatrix}
		1 & x_2^\intercal & u_2^\intercal \\
		1 & x_3^\intercal & u_3^\intercal \\
		\vdots & \vdots & \vdots\\
		1 & x_T^\intercal & u_T^\intercal \\
		1 & x_{T+1}^\intercal & u_{T+1}^\intercal 
	\end{pmatrix}
	\]

	Suppose that  the regularization term $R$ is given by 
	$$
	R(\w) = 
	\|w\|_F^2  + |w'|^2 
	$$

	The optimization problem of training a two layer neural network becomes
	
	\begin{equation}
	\label{2layermin}
	\begin{aligned}
			& \min_{w, w'} \left|(Xw)_+w' - y\right|^2 +  \rho_T \left(\|w\|_F^2 + 	|w'|^2\right) =\\
			& \min_{w, w'} \left|\sum_{i=1}^{M}(Xw_i)_+w_i' - y\right|^2 +  \rho_T \sum_{i=1}^{M}\left(|w_i|^2 + 	|w_i'|^2\right)
	\end{aligned}
	\end{equation}
	where 
	$$w = [w_1~w_2~ \cdots ~w_M],$$
	$$w' = 
	[w'_1~w'_2~ \cdots ~w'_M]^\T,$$ 
	$w_i\in \R^{d}$, $w'_i\in \R$, for $i=1, ..., M$. 
	
%

	Using the framework by Pilanci \textit{et. al.} \cite{pilanci:2020}, the optimization problem given by (\ref{2layermin}) can be transformed to an equivalent convex problem. The authors make the observation that $(Xu)_+ = DXu$ for some diagonal matrix $D$, with zeros and ones on its diagonal. By generating a set of $D$ matrices for each episode, and then optimize for the weights $\mathbf{w}$, we can get and equivalent convex optimization problem for training the weights of the neural network.

	The equivalent convex optimization problem is given by
	\begin{equation}
		\label{cvxopt0}
		\begin{aligned}
			\min_{w\in \mathcal{W}} 
			&~~ \left|\sum_{p=1}^{P}D_p X(w_{1,p}-w_{2,p}) - y\right|^2\\ &\hspace{1mm}+\rho_T \sum_{p=1}^{P}(|w_{1,p}| + |w_{2,p}|) \\
		\end{aligned}
	\end{equation}
	where $\mathcal{W}$ is a linear space given by 
		\[
		\begin{aligned}
			\mathcal{W} &= \left\{w:  ~ 0 \le (2D_p-1)Xw_{p},~ p = 1, ..., P  
			\right\}
		\end{aligned}
		\]
	for some positive integer $P$, which is known to be bounded by $\mathcal{O}((\frac{n}{r})^r)$, $r = \textup{rank}(X)$, and  
	$w = (w_1, ..., w_P)\in \mathbb{R}^{2(m+n+1)\times P}$. Once we solve for the weights $w$ in (\ref{cvxopt0}), we get the equivalent neural network  \cite{pilanci:2020}
	\begin{equation}
		Q(x,u,w) = \sum_{p=1}^{P}  (w_{1,p}^\T x)_+ - ( w_{2,p}^\T x)_+
	\end{equation} 
	
	Note that the optimization problem in (\ref{cvxopt0}) is equivalent to 
	\begin{equation}
		\label{cvxopt1}
		\begin{aligned}
			\min_{\substack{w\\ v > 0}} 
			&~~ \left|\sum_{p=1}^{P}D_p X(v_{1,p}w_{1,p}-v_{2,p}w_{2,p}) - y\right|^2\\
			&~+\rho_T \sum_{p=1}^{P}(|v_{1,p}w_{1,p}| + |v_{2,p}w_{2,p}|) \\
			\textup{subject to} 
			&~~ 0 \le (2D_p-1)Xw_{1, p},
			\hspace{2mm} p=1, ..., P\\
			&~~ 0 \le (2D_p-1)Xw_{2, p},
			\hspace{2mm} p=1, ..., P.
		\end{aligned}
	\end{equation}
	
	The learning algorithm we propose is given by Algorithm 1. 
	The main theoretical result of the paper provides conditions for which Algorithm 1 converges. Before stating our theorem, let $v_{i,p}$ be the optimal solution in (\ref{cvxopt1}), for $p = 1, .., P$, $i=1,2$, and introduce
	\[
	F_1 = \begin{bmatrix}
		\frac{\rho}{2v_{1,1}^2} & 0 & \cdots & 0\\
		0 & \frac{\rho}{2v^2_{1,2}} & \cdots & 0\\
		\vdots & \vdots & \ddots & \vdots\\
		0 & 0 & \cdots & \frac{\rho}{2v_{1,p}^2}
	\end{bmatrix},
	\]
	\[
	F_2 = \begin{bmatrix}
		\frac{\rho}{2v_{2,1}^2} & 0 & \cdots & 0\\
		0 & \frac{\rho}{2v^2_{2,2}} & \cdots & 0\\
		\vdots & \vdots & \ddots & \vdots\\
		0 & 0 & \cdots & \frac{\rho}{2v_{2,p}^2}
	\end{bmatrix},
	\]
	and
	\[
	F = \begin{bmatrix}
		F_1 & 0\\
		0 & F_2
	\end{bmatrix}.
	\]
	Now we are ready to state our main result.

	\begin{theorem}
		\label{main}
		Consider Algorithm 1 and let $\lambda$ and $\beta$ be positive real numbers such that 
		\[
		\begin{aligned}
			\lambda I \preceq \frac{1}{T}\left(
			\sum_{t=1}^{T}
			\begin{bmatrix}
					X_t^\T\\
					-X_t^\T
			\end{bmatrix}
			\begin{bmatrix}
				X_t &
				-X_t
			\end{bmatrix} + F \right)
		\end{aligned}
		\]
		and
		\[
			 |u_t|^2\le \beta,~~ |x_t|^2 \le \beta, ~~~~\textup{for } t = 1, ..., T.
		\]
		 Further, let $\alpha_k$ be either a positive real constant, $\alpha_k = \alpha < 1$ for all $k\in \mathbb{N}$, or such that $\alpha_k = 1/k$.
		 If 
		\[
			T > \frac{3}{2}\left(\frac{\gamma \beta}{\lambda}\right)^2,
		\]
		then for some positive real constant $C$, $w_k\rightarrow w$ where 
		\[
		\|w - w^\star\| \le C \rho
		\] 
		and $w^\star$ are the optimal neural network parameters satisfying the Bellman equation
	\begin{equation*}
		\begin{aligned}
			Q(x,u, w^\star) 	  &= c(x,u) + \pr_{\mathbf{Q}}\left(\gamma \min_{u_+} Q(x_+, u_+, w^\star)\right)
		\end{aligned}		
	\end{equation*}
		\end{theorem}
	\begin{proof} The proof of this theorem is provided in a later section.

\begin{remark} 
	Theorem \ref{main} shows that if we consider a long enough horizon (which depends on the problem parameters), we get arbitrarily close to the optimal parameters when we let $\rho$ go to zero. The dependency on $\rho$ is surprising since the regularization causes our solution to deviate from the optimal parameters, where they would coincide if $\rho = 0$. The regularization parameter is however strictly positive since this assumption is needed in order to get the equivalent transformation to a convex optimization problem.
\end{remark}

	\begin{remark} 
			Note that Algorithm 1 above is useful for both cases when the system parameters are known as well as when they are unknown. In the case of known parameters, it's hard in general to have a closed form solution for the controller. Training a $Q$-function based on convex optimization is practical way to find a controller that is near optimal. In the case of unknown system parameters and cost function, we can just use the measured values of $c(x,u)$ and the state $x$. It would be interesting to compare both results when assuming known and unknown system parameters.  
	\end{remark}
		
\section{Numerical Experiment}
In this section we will consider reinforcement learning for nonlinear dynamical system given by
\[
	x_{t+1} = 0.9x_t^2 + 0.1u_t
\]	
with a cost function given by $$c(x,u) = x^2 + (0.1u-2x)^2$$
and constraint $|u_t|\le 5$. for simplicity, We restrict the training to initial states $x_0\in [0,1]$. 
Since it's numerically (NP-) hard to find the optimal stationary (time-invariant) controller, even if  we have full knowledge of the system model and cost function, we solve the finite horizon problem using dynamic programming and gridding of the state space applied to the Bellman equation. Then, we deduce the optimal time-varying controller. This would provide a lower bound to the cost using a stationary controller. In particular we verify the cost over a short horizon of 5 time steps, that is $T=5$. The number of episodes was set  $K=1000$ to train a two layer neural network that approximates the optimal $Q$-function, and the regularization parameter is set to $\rho = 10^{-4}$. Table \ref{table} summarizes the results for different initial states. We can see that the neural networks found by the convex optimization based episodic learning algorithm presented in this paper gets very close to the lower bound of the optimal controller.

\begin{table}[t]
	\label{table}
	\centering
	\begin{tabular}{lcc}
		\hline
		Initial State & Lower Bound & Convex Optimization \\ \hline
		0.25 & 0.346 & 0.364 \\
		0.50 & 1.493 & 1.548 \\
		0.75 &  0.315 &  0.364\\
		1.00 &  8.140 &  9.981\\ \hline
	\end{tabular}
	\caption{The table shows the performance of the trained neural network, compared to the lower bound given by the optimal finite horizon controller found by numerically solving the Bellman equation. We considered a time horizon $T=5$ for the time-varying controller, and tested different initial states $x_0$ in the interval $[0,1]$. We see clearly that the convex optimization approach presented in this paper gets very close to the optimal solution after $1000$ episodes.}
	\label{tab:your_label}
\end{table}

\section{Proof of Theorem \ref{main}}	
First note that
\[
y = c + \sum_{p=1}^{P}D'_p Z(w_{1,p,k}-w_{2,p,k})
\] 
where $w_k$ are the parameters used in episode $k$ and
\[
[D'_p]_{tt} =  
\left\{
\begin{array}{ll}
	[D_p]_{(t+1)(t+1)}  & \textup{for } 1\le t \le T-1 \\
	0					&  \textup{for } t = T
\end{array}
\right. \\
\]

Let $\mathcal{Q}$ be the space of functions given by 
\[
\begin{aligned}
	&\mathcal{Q} = \Big\{Q:~  Q = \sum_{p=1}^{P}D_p X(w_{1,p}-w_{2,p}), \\
	&\hspace{1cm} 
	 ~ 0 \le (2D_p-1)Xw_{i,p}, i=1,2
	\Big\}.
\end{aligned}
\]
The optimal weights $w^\star$ must satisfy the Bellman equation
\begin{equation*}
	\label{bellman}
	\begin{aligned}
		&\sum_{p=1}^{P}D_p X(w^\star_{1,p}-w^\star_{2,p}) =\\
		& c +  
		\pr_{\mathcal{Q}}\left(\gamma \min_u\sum_{p=1}^{P}D'_p Z(w^\star_{1,p}-w^\star_{2,p})\right)
	\end{aligned}
\end{equation*}
%
%

Since 
$$
|w_{i,p}v_{i,p}| \le \frac{1}{2}\left(|w_{i,p}|^2 + v_{i,p}^2\right),
$$
with equality if and only if $|w_{p}| = v_p$, we see that (\ref{cvxopt1}) is equivalent to
\begin{equation}
	\label{cvxopt2}
	\begin{aligned}
		\min_{\substack{w\\ v > 0}} 
		&~~ \left|\sum_{p=1}^{P}D_p X(v_{1,p}w_{1,p}-v_{2,p}w_{2,p}) - y\right|^2\\
		&~~~~ + \frac{\rho_T}{2}\sum_{p=1}^{P}\frac{1}{2}\left(|w_{i,p}|^2 + v_{i,p}^2\right)\\
		\textup{subject to} 
		&~ 0 \le (2D_p-1)Xw_{i,p},\hspace{2mm} p=1, ..., P,~~i=1,2
	\end{aligned}
\end{equation}
Now the transformation $w_{i,p}v_{i,p}\rightarrow w_{i,p}$ implies the equivalent optimization problem
\begin{equation}
	\label{cvxopt3}
	\begin{aligned}
		\min_{v > 0} \min_{w} 
		&~~ \left|\sum_{p=1}^{P}D_p X(w_{1,p}-w_{2,p})-y\right|^2\\
		&~ + \frac{\rho_T}{2}\sum_{p=1}^{P}\left(|w_{i,p}/v_{i,p}|^2 + v_{i,p}^2\right) \\
		\textup{subject to} 
		&~~ 0 \le (2D_p-1)Xw_{1, p},
		\hspace{2mm} p=1, ..., P\\
		&~~ 0 \le (2D_p-1)Xw_{2, p},
		\hspace{2mm} p=1, ..., P.
	\end{aligned}
\end{equation}

Introduce 
\[
\begin{aligned}
	D_X &= \left[D_1X~ D_2X~ \cdots ~ D_pX\right]\\
	D_Z &= \left[D'_1Z~ D'_2Z~ \cdots ~ D'_pZ\right],		
\end{aligned}
\]

\[
H = 
\begin{bmatrix}
	I 			 & -D_X							& D_X\\ 
	-D^\T _X	& D_X^\T D_X + F_1 	 	  & -D_X^\T D_X \\
	D^\T _X 	& -D_X^\T D_X & D_X^\T D_X + F_2
\end{bmatrix},
\]
and
\[
\begin{aligned}
	\Lambda &= 
	\begin{bmatrix}
		D_X^\T D_X + F_1	& -D_X^\T D_X \\ 
		-D_X^\T D_X 			 & D_X^\T D_X + F_2
	\end{bmatrix}\\
	L &= \Lambda^{-1} 
	\begin{bmatrix}
		D^\T _X\\ 
		-D^\T _X 
	\end{bmatrix},\\
	M &= I - L^\T \Lambda^{-1} L.
\end{aligned}
\]
The objective	function in (\ref{cvxopt3}) is given by
\[
\begin{aligned}
	&\left|\sum_{p=1}^{P}D_p X(w_{1,p}-w_{2,p})-y\right|^2\\
	& + \frac{\rho_T}{2}\sum_{p=1}^{P}\left(|w_{1,p}/v_{1,p}|^2 + |w_{2,p}/v_{2,p}|^2 + v_{1,p}^2 + v_{2,p}^2\right) \\
	&= \left|\sum_{p=1}^{P}D_p X(w_{1,p}-w_{2,p}) - c \right. \\
	&~~~~\left. -\gamma \sum_{p=1}^{P}D'_p Z(w^k_{1,p}-w^k_{2,p}) \right|^2\\
	& ~~~~+ \frac{\rho_T}{2}\sum_{p=1}^{P}\left(|w_{1,p}/v_{1,p}|^2 + |w_{2,p}/v_{2,p}|^2 + v_{1,p}^2 + v_{2,p}^2\right) \\
	&=  \left|\sum_{p=1}^{P}D_p X(w_{1,p}-w_{2,p}) \right.\\
	&~~~~ \left. - c - \pr_{\mathcal{Q}}\left( \gamma\sum_{p=1}^{P}D'_p Z(w^k_{1,p}-w^k_{2,p})\right) \right|^2\\
	&~~~~+ \left|(I-\pr_{\mathcal{Q}})\left(\gamma\sum_{p=1}^{P}D'_p Z(w^k_{1,p}-w^k_{2,p})\right)\right|^2\\
	&~~~~ + w_1^\T F_1 w_1 + w_2^\T F_2 w_2 + \frac{\rho_T}{2}\sum_{p=1}^{P} |v_p|^2\\
	&= \left| D_X(w_1-w_2) - \bar{y}\right|^2\\
	&~~~~ +\left|(I-\pr_{\mathcal{Q}})\left(\gamma\sum_{p=1}^{P}D'_p Z(w^k_{1,p}-w^k_{2,p})\right)\right|^2\\
	&~~~~ + w_1^\T F_1 w_1 + w_2^\T F_2 w_2 + \frac{\rho_T}{2}\sum_{p=1}^{P} |v_p|^2\\
	&= 	\begin{bmatrix}
		\bar{y}\\ 
		w_1\\
		w_2
	\end{bmatrix}^\T
	H
	\begin{bmatrix}
		\bar{y}\\
		w_1\\
		w_2
	\end{bmatrix}	+  \frac{\rho}{2}\sum_{p=1}^{P} v_p^2 \\
\end{aligned}
\]

Using the Bellman equation (\ref{bellman}), we get
\[
\begin{aligned}
	\bar{y} &= c + \pr_{\mathcal{Q}}(\gamma D_Z(w_{1,k} - w_{2,k}))\\
	&= D_X(w_1^\star-w_2^\star) -  \pr_{\mathcal{Q}}(\gamma D_Z(w_1^\star-w_2^\star))\\ 
	&~~~~+ \pr_{\mathcal{Q}}(\gamma D_Z(w_{1,k} - w_{2,k}))\\
	&= [D_X~~ -D_X]w^\star + \gamma \pr_{\mathcal{Q}}([D_Z~~ -D_Z]w_k)\\
	&~~~~- \pr_{\mathcal{Q}}(\gamma[D_Z~~ -D_Z]w^\star)
\end{aligned}
\]
Then, completion of squares gives the relation
\[
\begin{aligned}
	&\begin{bmatrix}
		\bar{y}\\ 
		w
	\end{bmatrix}^\T
	H
	\begin{bmatrix}
		\bar{y}\\ 
		w
	\end{bmatrix} =
	\left(w - L\bar{y}
	\right)^\T  
	\Lambda
	\left(w - L\bar{y}\right)+\bar{y}^\T M\bar{y}
\end{aligned}
\]

Since $\Lambda \prec 0$, we may define the norm
\[
\|w\|_{\Lambda} \triangleq w^\T \Lambda w.
\]

Let $\mathcal{W}$ be the linear space given by 
\[
\begin{aligned}
	\mathcal{W} &= \left\{w: ~ 0 \le (2D_p-1)Xw_{p},~ p = 1, ..., P\right\}.
\end{aligned}
\]
Standard Hilbert Space Theory implies that the optimal solution $w^\star$ to 
\[
\min_{w\in \mathcal{W}} \|w-L\bar{y}\|_{\Lambda}
\]
is the projection of $Lv$ on $\mathcal{W}$ under the norm $\|\cdot \|_{\Lambda}$, that is
\[
w_k^\star = \pr_{\mathcal{W}}(L\bar{y}).
\]

Thus,
\[
\begin{aligned}
	w^\star_k &= 
	\pr_{\mathcal{W}}\Big(\Lambda^{-1}\begin{bmatrix}
		D^\T _X\\ 
		-D^\T _X 
	\end{bmatrix}\big( c  +  \pr_{\mathcal{Q}}\left(\gamma[D_Z~~ -D_Z]w_k\right) \big)\Big)
\end{aligned}
\]
and
{\small
	\[
	\begin{aligned}
		w^\star_k  &- w^\star\\
		&= \pr_{\mathcal{W}}\Big(\Lambda^{-1}
		\begin{bmatrix}
			D^\T _X\\ 
			-D^\T _X 
		\end{bmatrix}
		\big( c  + \pr_{\mathcal{Q}}\left(\gamma [D_Z~~ -D_Z]w_k\right)\big)  - w^\star\Big) \\
		&= \pr_{\mathcal{W}}\Big(\Lambda^{-1}
		\begin{bmatrix}
			D^\T _X\\ 
			-D^\T _X 
		\end{bmatrix}
		\big( c  + \pr_{\mathcal{Q}}\left(\gamma [D_Z~~ -D_Z]w_k\right)\big)\\
		&~~~~  - \Lambda^{-1}\left(
		\begin{bmatrix}
			D^\T _X\\ 
			-D^\T _X 
		\end{bmatrix}
		\begin{bmatrix}
			D_X   & -D_X
		\end{bmatrix}
		+ F
		\right)
		w^\star\Big) \\
		&= \pr_{\mathcal{W}}\Big(\Lambda^{-1}
		\begin{bmatrix}
			D^\T _X\\ 
			-D^\T _X 
		\end{bmatrix}
		\big(c + \pr_{\mathcal{Q}}\left(\gamma[D_Z~~ -D_Z]w_k\right)\\ 
		&~~~~ - [D_X~~ -D_X]w^\star\big) - \Lambda^{-1}Fw^\star \Big)\\
		&= \pr_{\mathcal{W}}\Big(\Lambda^{-1}
		\begin{bmatrix}
			D^\T _X\\ 
			-D^\T _X 
		\end{bmatrix}
		\big(\pr_{\mathcal{Q}}\left(\gamma[D_Z~~ -D_Z]\big(w_k-w^\star\big)\right)\\ 
		&\hspace{1.5cm} - \Lambda^{-1}Fw^\star \Big)
	\end{aligned}
	\]
}

Now the update rule for $w_k$ implies that
\[
w_{k+1} - w^\star  = (1-\alpha_k)(w_k - w^\star ) + \alpha_k (w_k^\star - w^\star ) 
\]
and
\[
\begin{aligned}
	&w_{k+1} - w^\star\\
	& = (1-\alpha_k)(w_k - w^\star ) + \alpha_k (w_k^\star - w^\star ) \\
	&= (1-\alpha_k )(w_k - w^\star ) \\ 
	&+ \alpha_k  \pr_{\mathcal{W}}\Big(\Lambda^{-1}
	\begin{bmatrix}
		D^\T _X\\ 
		-D^\T _X 
	\end{bmatrix}
	\big(\pr_{\mathcal{Q}}\left(\gamma[D_Z~~ -D_Z]\big(w_k-w^\star\big)\right)\\ 
	&\hspace{1.5cm} - \Lambda^{-1}Fw^\star \Big)
\end{aligned}
\]

Introduce
\[
A = 	
\begin{bmatrix}
	D_X   & -D_X
\end{bmatrix},
\]
Then,
\[
\begin{aligned}
	A^\T A &=
	\begin{bmatrix}
		D^\T_X \\ 
		-D^\T_X
	\end{bmatrix}
	\begin{bmatrix}
		D_X   & -D_X
	\end{bmatrix} \\
	&= 
	\begin{bmatrix}
		D_X^\T   D_X & -D_X^\T   D_X\\
		-D_X^\T   D_X & D_X^\T   D_X
	\end{bmatrix} 
\end{aligned}
\]
\[
\begin{aligned}
	A A^\T &=
	\begin{bmatrix}
		D_X   & -D_X
	\end{bmatrix}
	\begin{bmatrix}
		D^\T_X \\ 
		-D^\T_X
	\end{bmatrix} \\
	&= 2  D_X D_X^\T  
\end{aligned}
\]
Also, note that	
\[
\begin{aligned}
	A^\T  A + F
	&= \sum_{t=1}^T\sum_{p=1}^{P}[D_p]_{tt}
	\begin{bmatrix}
		X_t^\T\\
		-X_t^\T
	\end{bmatrix}
	\begin{bmatrix}
		X_t & -X_t
	\end{bmatrix} + F\\
	&\succeq  
	\sum_{t=1}^T\sum_{p=1}^{P}
	\begin{bmatrix}
		X_t^\T\\
		-X_t^\T
	\end{bmatrix}
	\begin{bmatrix}
		X_t &
		-X_t
	\end{bmatrix} + F\\
	&\succeq 	T\cdot \lambda I 
\end{aligned}
\]

and
\begin{equation}
	\begin{aligned}
		\left\|\Lambda^{-1}F\right\| 
		&=\left\|(A^\T A+F)^{-1}F\right\| \\
		&\le \left\|(T\cdot\lambda I )^{-1}F\right\|\\
		&\le \left\|(T\cdot\lambda I)^{-1}\rho_T f_{\textup{max}}\right\|\\
		&\le \frac{\rho_T f_{\textup{max}}}{\lambda T}\\
		&= \frac{\rho f_{\textup{max}}}{\lambda}
	\end{aligned}
\end{equation}

Now we have that	
\[
\begin{aligned}
	(A^\T A + F)^{-1}A^\T & A(A^\T A + F)^{-1} \preceq (A^\T A + F)^{-1}
\end{aligned}
\]

Introduce 
\[
\begin{aligned}
	f_{\textup{min}} &= \min_{i,p} \left\{\frac{1}{v_{i,p}}\right\}\\  
	f_{\textup{max}} &= \max_{i,p} \left\{\frac{1}{v_{i,p}}\right\},
\end{aligned}
\] 
and note that for any $f\in (0, \infty)$, we have the relation
\[
\begin{aligned}
	(A^\T A + fI)^{-1}A^\T &=  A^\T(A A^\T  + fI)^{-1} 
\end{aligned}
\]
Thus,
\begin{equation}
	\label{sys}
	\begin{aligned}
		\left\|\Lambda^{-1}
		\begin{bmatrix}
			D^\T _X\\ 
			-D^\T _X 
		\end{bmatrix}
		\right\|
		&=\left\| (A^\T A+F)^{-1}A^\T \right\| \\
		& \le\left\|(A^\T A+\rho_T f_{\textup{min}}I)^{-1}A^\T\right\|  \\
		&= \left\|A^\T(AA^\T +\rho_T  f_{\textup{min}}I)^{-1}\right\|  \\
		&= \left\|
		\begin{bmatrix}
			D^\T _X\\ 
			-D^\T _X 
		\end{bmatrix}
		(2D_X D_X^\T + \rho_T  f_{\textup{min}}I)^{-1}
		\right\| \\
		&= \left\|
		D^\T _X (2D_X D_X^\T + \rho_T  f_{\textup{min}}I)^{-1}
		\right\| \\
		&= \left\|
		(2D_X^\T D_X  + \rho_T  f_{\textup{min}}I)^{-1}D^\T _X
		\right\|  \\
		&\le \left\|
		(2D_X^\T D_X  + \rho_T  f_{\textup{min}}I)^{-1}
		\right\| \|D_X^\T\| \\
		&\le\left\|(2T\cdot\lambda I+\rho_T  f_{\textup{min}}I)^{-1}\right\| \cdot \sqrt{T\beta} \\
		&=\frac{\sqrt{T\beta}}{2T\lambda + \rho_T  f_{\textup{min}}}\\
		&\le \sqrt{\frac{\beta}{4\lambda^2 T}}\\
	\end{aligned}
\end{equation}
Thus,
{\small
	\[
	\begin{aligned}
		&\|w_{k+1} - w^\star\| \\
		&\le (1-\alpha_k )\|w_k - w^\star\|\\ 
		&+ \alpha_k  \left\|\pr_{\mathcal{W}}\Big( \Lambda^{-1}
		\begin{bmatrix}
			D^\T _X\\ 
			-D^\T _X 
		\end{bmatrix}
		\pr_{\mathcal{Q}}\left(\gamma[D_Z~~ -D_Z]\big(w_k-w^\star\big)\right)\Big) \right\|\\
		&+ \alpha_k \left\|\pr_{\mathcal{W}}\left(\Lambda^{-1} F w^\star\right) \right\|\\
		&\le (1-\alpha_k )\|w_k - w^\star\|\\ 
		&+ \alpha_k  \left\|\pr_{\mathcal{W}}\Big( \Lambda^{-1}
		\begin{bmatrix}
			D^\T _X\\ 
			-D^\T _X 
		\end{bmatrix}\Big)\right\|
		\left\|\pr_{\mathcal{Q}}\left(\gamma[D_Z~~ -D_Z]\right)\right\| \left\|w_k-w^\star \right\|\\
		&  + \alpha_k \left\|\pr_{\mathcal{W}}\left(\Lambda^{-1} F w^\star\right) \right\|\\
		&\le (1-\alpha_k )\|w_k - w^\star\|\\ 
		&+ \alpha_k  \left\|\Lambda^{-1}
		\begin{bmatrix}
			D^\T _X\\ 
			-D^\T _X 
		\end{bmatrix}\right\|
		\left\|\gamma[D_Z~~ -D_Z]\right\| \left\|w_k-w^\star \right\|\\
		&  + \alpha_k \left\|\Lambda^{-1} F w^\star \right\|\\
		&\le (1-\alpha_k )\|w_k - w^\star\|\\ 
		&+ \alpha_k  (\gamma \sqrt{6\beta}) \left\|\Lambda^{-1}
		\begin{bmatrix}
			D^\T _X\\ 
			-D^\T _X 
		\end{bmatrix}\right\|
		\left\|w_k-w^\star \right\|\\
		&  + \alpha_k \left\|\Lambda^{-1} F w^\star \right\|\\
		&\le (1-\alpha_k )\|w_k - w^\star\| + \alpha_k  (\gamma \sqrt{6\beta}) \sqrt{\frac{\beta}{4\lambda^2 T}}
		\left\|w_k-w^\star \right\|\\
		&  + \alpha_k \left\|\Lambda^{-1} F w^\star \right\|\\
		&= (1-\alpha_k )\|w_k - w^\star\| + \alpha_k  
		\sqrt{ \frac{3\gamma^2 \beta^2}{2\lambda^2 T}}
		\left\|w_k-w^\star \right\|\\
		&  + \alpha_k \left\|\Lambda^{-1} F w^\star \right\|\\
		&\le \left(1-\alpha_k  + \alpha_k  
		\sqrt{ \frac{3\gamma^2 \beta^2}{2\lambda^2 T}}\right)
		\left\|w_k-w^\star \right\|
		+ \alpha_k \frac{\rho_T f_{\textup{max}}}{\lambda T}\|w^\star\|\\
		&= \left(1-\alpha_k  + \alpha_k  
		\sqrt{ \frac{3\gamma^2 \beta^2}{2\lambda^2 T}}\right)
		\left\|w_k-w^\star \right\|
		+ \alpha_k \frac{\rho f_{\textup{max}}}{\lambda}\|w^\star\|  
	\end{aligned}
	\]
}	
The inequality
\[
T > \frac{3}{2}\left(\frac{\gamma \beta}{\lambda}\right)^2
\]
implies that
\[
\frac{3\gamma^2 \beta^2}{2\lambda^2 T} < 1
\]
Introduce 
\[
\mu = 1- \sqrt{\frac{3\gamma^2 \beta^2}{2\lambda^2 T}} < 1
\]
Since $\alpha_k \le 1$, we have that
$$1-\mu \alpha_k  < 1.$$
For the case where $\alpha_k = \alpha$, for all $k\in \mathbb{N}$, we see that the inequality in (\ref{sys}) implies that $w_k-w^\star$ converges to some $w - w^\star$ where
\[
\|w-w^\star\| \le \frac{\mu \alpha}{1-\mu \alpha}\frac{\rho f_{\textup{max}}}{\lambda}\|w^\star\|  			
\]
Similarly, for $\alpha_k = \frac{1}{k}$,  we have that $\|w_k - w^\star\|$ is the state $ \Delta_{k}$ of a stable dynamical system given by
\[
\Delta_{k+1} \le (1-\mu\alpha_k)\cdot \Delta_{k} + \alpha_k \frac{\rho f_{\textup{max}}}{\lambda}\|w^\star\| 
\]
More explicitely, we have that
\begin{equation}
	\begin{aligned}
		\Delta_K \le  \prod_{k=1}^{K} (1-\mu \alpha_k)\Delta_0 + \sum_{k=1}^K \prod_{i=k+1}^{K} (1-\mu \alpha_i)\alpha_k \frac{\rho f_{\textup{max}}}{\lambda}\|w^\star\| 
	\end{aligned}
\end{equation}
Note that
\begin{equation}
	\begin{aligned}
		\prod_{k=1}^K (1 - \mu\alpha_k) 
		& \le \prod_{k=1}^K \exp(- \mu\alpha_k)\\
		&= \exp(- \mu \sum_{k=1}^nK\alpha_k)\\
		&\le \exp(- \mu \ln{K})\\
		&= \frac{1}{K^\mu}
	\end{aligned}
\end{equation}
which goes to zero as $K\rightarrow \infty$.
Similarly, for $k > 0$, we have
\begin{equation}
	\begin{aligned}
		\prod_{i=k+1}^K (1 - \mu\alpha_i) 
		& \le \prod_{i=k+1}^K \exp(- \mu\alpha_i)\\
		&= \exp \left(- \mu \sum_{i=k+1}^K \alpha_i\right)\\
		&\le \exp \left(- \mu (\ln{n} - \ln k  - 1)\right)\\
		&= \left(\frac{ek}{K}\right)^\mu
	\end{aligned}
\end{equation}
Thus,
\begin{equation}
	\begin{aligned}
		\sum_{k=1}^{K}\prod_{i=k+1}^K (1 - \mu\alpha_i)\alpha_k  
		& \le \sum_{k=1}^K \left(\frac{ek}{K}\right)^\mu \frac{1}{k}\\
		& = \left(\frac{e}{K}\right)^\mu \sum_{k=1}^K \left(\frac{1}{k}\right)^{1-\mu}\\
		&\le \left(\frac{e}{K}\right)^\mu \int_{1}^{K+1} x^{\mu-1} dx \\
		&= \left(\frac{e}{K}\right)^\mu \frac{1}{\mu} \left( (K+1)^{\mu} - 1\right) \\
	\end{aligned}
\end{equation}
Since
\[
\lim_{K\rightarrow \infty} \frac{e^\mu}{\mu} \frac{1}{K^\mu} \left( (K+1)^{\mu} - 1\right)  =  \frac{e^\mu}{\mu},
\]
we get the inequality 
\begin{equation}
	\begin{aligned}
		\sum_{k=1}^{\infty}\prod_{i=k+1}^\infty (1 - \mu\alpha_i)\alpha_k  
		& \le  \frac{e^\mu}{\mu}.
	\end{aligned}
\end{equation}

We conclude that
\[
\|w - w^\star\| = \Delta_{\infty} \le C\cdot {\rho}
\] 
with $$C = \frac{e^\mu}{\mu}\frac{f_{\textup{max}}}{\lambda }\|w^\star\|.$$ 

This completes the proof.

\end{proof}

	\section{Conclusions}
	We have considered the problem of reinforcement learning for optimal control of stable nonlinear systems in an episodic setting, where we at each episode  approximate the $Q$-function with a two-layer neural network for which the optimal parameters per episode are found by using convex optimization. We show that as the number of episodes goes to infinity, the algorithm converges to neural network parameters given by $w$ such that the distance to the optimal network parameters $w^\star$ is bounded according to the inequality
	\[
	\|w - w^\star\| \le C{\rho},
	\]  
	for some constant $C$. In particular, we can see that as regularization parameter $\rho$ decreases, the converging neural network parameters get arbitrarily close to the optimal ones. As a consequence, our algorithm converges fast due to the polynomial-time convergence of convex optimization algorithms.
	
	Future work includes applications to reinforcement learning for mixed continuous and discrete state and action spaces. It would also be interesting to apply the algorithm to fine-tuning Large Language Models, since the training would be fast and computationally very efficient, and convergence to the optimal neural network parameters is guaranteed.   
	\bibliographystyle{icml2023}
	\bibliography{../../../ref/mybib}

\end{document}